\documentclass[11pt]{article}

\usepackage{amsfonts,amsmath,amssymb}
\usepackage{amsthm}
\usepackage{graphicx}

\newtheorem{theorem}{Theorem}[section]
\newtheorem{thm}[theorem]{Theorem}
\newtheorem{prop}[theorem]{Proposition}
\newtheorem{lem}[theorem]{Lemma}
\newtheorem{claim}[theorem]{Claim}

\newtheorem{obs}[theorem]{Observation}
\newtheorem{defn}[theorem]{Definition}

\newtheorem{conjecture}[theorem]{Conjecture}

\newcommand{\beq}[1]{\begin{equation}\label{#1}}
\newcommand{\enq}[0]{\end{equation}}

\newcommand{\f}[0]{{\cal F}}

\newcommand{\p}[0]{{\cal P}}
\newcommand{\C}[0]{{\cal C}}

\newcommand{\pn}[0]{{\cal P}([n])}

\begin{document}
\global\long\def\f{\mathcal{F}}
\global\long\def\l{\mathcal{L}}
\global\long\def\pn{\mathcal{P}\left(\left[n\right]\right)}
\global\long\def\g{\mathcal{G}}
\global\long\def\s{\mathcal{S}}
\global\long\def\m{\mathcal{M}}
\global\long\def\mh{\mu_{\frac{1}{2}}}
\global\long\def\mp{\mu_{p}}
\global\long\def\j{\mathcal{J}}
\global\long\def\d{\mathcal{D}}
\global\long\def\Inf{\mathrm{Inf}}
\global\long\def\p{\mathcal{P}}
\global\long\def\mpo{\mu_{p_{0}}}
\global\long\def\fpp{f_{p}^{p_{0}}}
\global\long\def\llll{\l_{\mu}}
\global\long\def\h{\mathcal{H}}
\global\long\def\n{\mathbb{N}}
\global\long\def\a{\mathcal{A}}
\global\long\def\b{\mathcal{B}}
\global\long\def\sf{f_{2}}
\global\long\def\bin{\mathrm{Bin}}
\global\long\def\C{C_{2}}

\title{Approximation of Biased Boolean Functions of Small Total Influence by DNF's}

\author{
Nathan Keller\thanks{Department of Mathematics, Bar Ilan University, Ramat Gan, Israel.
{\tt nathan.keller27@gmail.com}. Research supported by the Israel Science Foundation (grant no.
402/13) and by the Binational US-Israel Science Foundation (grant no. 2014290).}
\mbox{ } and Noam Lifshitz\thanks{Department of Mathematics, Bar Ilan University, Ramat Gan, Israel.
{\tt noamlifshitz@gmail.com}.}
}

\maketitle

\begin{abstract}
The influence of the $k$'th coordinate on a Boolean function $f:\{0,1\}^n \rightarrow \{0,1\}$ is the probability that
flipping $x_k$ changes the value $f(x)$. The total influence $I(f)$ is the sum of influences of the
coordinates. The well-known `Junta Theorem' of Friedgut (1998) asserts that if $I(f) \leq M$, then $f$ can be $\epsilon$-approximated by
a function that depends on $O(2^{M/\epsilon})$ coordinates. Friedgut's theorem has a wide variety of applications in mathematics
and theoretical computer science.

For a biased function with $\mathbb{E}[f]=\mu$, the edge isoperimetric inequality on the cube implies that $I(f) \geq 2\mu \log(1/\mu)$.
Kahn and Kalai (2006) asked, in the spirit of the Junta theorem, whether any $f$ such that $I(f)$ is within a constant factor of the
minimum, can be $\epsilon \mu$-approximated by a DNF of a `small' size (i.e., a union of a small number of sub-cubes). We answer the question
by proving the following structure theorem: If $I(f) \leq 2\mu(\log(1/\mu)+M)$, then $f$ can be $\epsilon \mu$-approximated by a DNF of size $2^{2^{O(M/\epsilon)}}$. The dependence on $M$ is sharp up to the constant factor in the double exponent.
\end{abstract}

\section{Introduction}

\subsection{Background}

Let $f$ be a Boolean function on the discrete cube, that is, $f:\{0,1\}^n \rightarrow \{0,1\}$. The \emph{influence} of the $k$'th coordinate on $f$ is
\[
I_k(f)=\Pr[f(x) \neq f(x \oplus e_k)],
\]
where $x \oplus e_k$ is obtained from $x$ by flipping the the $k$'th coordinate and leaving the other coordinates unchanged. The \emph{total influence} (or, in short, the \emph{influence}) of $f$ is defined as $I(f)=\sum_{k=1}^n I_k(f)$.

The notion of influences appears naturally in many contexts, such as isoperimetric inequalities (as $I(f)$ equals, up to normalization, to the \emph{edge boundary} of the subset $\{x:f(x)=1\}$ of the discrete cube), threshold phenomena in random graphs, cryptographic properties of election functions, etc. As a result, the last three decades witnessed a very extensive study of the `theory of influences', that has led to numerous applications in areas as diverse as theoretical computer science (e.g., hardness of approximation~\cite{Dinur-Safra,Hastad} and machine learning~\cite{OS07}), percolation theory~\cite{BKS}, social choice theory~\cite{Mossel-Arrow}, and others (see the survey~\cite{Kalai-Safra}).

The minimal possible value of the total influence, as function of the expectation $\mathbb{E}[f]$, can be derived from the classical \emph{edge isoperimetric inequality on the cube}~\cite{Bernstein,Harper,Hart,Lindsay}, which asserts that for any $m$, among all the $m$-element subsets of the discrete cube, the minimal edge boundary is attained by the set of the $m$ largest elements in the lexicographic order. A weaker (but more convenient and so more widely-used) bound is:
\begin{thm}[Harper, Bernstein, Lindsey, Hart]
\label{Thm:Iso-weak}
For any $f:\{0,1\}^n \rightarrow \{0,1\}$, we have
\[
I(f) \geq 2\mu(f) \log(1/\mu(f)),
\]
where $\mu(f):=\mathbb{E}[f]$. Equality is attained if and only if $f$ is a sub-cube.
\end{thm}

One of the best-known and most widely-used results on influences is Friedgut's `Junta Theorem'~\cite{Friedgut98} which describes the structure of functions with a low influence:
\begin{thm}[Friedgut]
\label{Thm:Friedgut-Junta}
Let $f:\{0,1\}^n \rightarrow \{0,1\}$ be a balanced Boolean function (i.e., $\mathbb{E}[f]=1/2$) that satisfies $I(f) \leq M$, and let $\epsilon>0$. Then there exists a Boolean function $g$ that $\epsilon$-approximates $f$ (i.e., $\Pr[f(x) \neq g(x)] \leq \epsilon$) such that $g$ depends on $2^{O(M/\epsilon)}$ coordinates. The dependence on $M$ is sharp, up to a multiplicative constant.
\end{thm}
For a balanced function $f$, Theorem~\ref{Thm:Iso-weak} implies that $I(f) \geq 1$. Hence, Theorem~\ref{Thm:Friedgut-Junta} may be viewed as a \emph{structure theorem} for balanced functions with influence within a constant multiplicative factor of the minimum possible.

\medskip

While balanced functions and the uniform measure on the discrete cube are sufficient for many of the applications of Theorem~\ref{Thm:Friedgut-Junta}, some applications -- most notably, to threshold phenomena in random graphs and other structures -- require to generalize the results to biased functions (i.e., $\mathbb{E}[f] \neq 1/2$), and to the setting of the \emph{biased measure} $\mu_p$ on the discrete cube, defined by $\mu_p(x)=p^{\sum x_i} (1-p)^{n-\sum x_i}$. Theorem~\ref{Thm:Friedgut-Junta} extends easily to these settings. However, the dependence of the results on $\mathbb{E}[f]$ (resp. on $p$) is such that they become much less informative when $\mathbb{E}[f]=o(1)$ or $\mathbb{E}[f]=1-o(1)$ (resp. $p=o(1)$ or $p=1-o(1)$), as the size of the approximating Junta $g$ becomes `too large'.

The case of balanced functions with respect to a biased measure was studied in numerous works and led to breakthrough results on the sharpness of thresholds of graph properties, such as the $k$-SAT problem (see Friedgut~\cite{Friedgut-SAT}, Bourgain~\cite{Bourgain99}, Bourgain-Kalai~\cite{Bourgain-Kalai}, and Hatami~\cite{Hatami12}). In a nutshell, it was shown that while influence within a constant factor of the minimum possible does not imply that the function can be approximated by a Junta, it allows to say that the function admits a weaker structure called in~\cite{Hatami12} `pseudo-Junta', and if it is `somewhat symmetric' then stronger structural properties hold~\cite{Bourgain-Kalai,Friedgut-SAT}.

The case of biased functions with a \emph{very low} influence was also studied in a number of works. Those works aimed at proving a \emph{stability version} of the edge isoperimetric inequality on the cube, asserting that if the influence of $f$ is within a small (additive) distance of the minimum possible, then $f$ is close (in the $\ell^1$ norm) to the indicator function of an extremal family. After a series of works which proved stability in specific cases (Friedgut, Kalai and Naor~\cite{FKN}, Bollob\'as, Leader and Riordan (unpublished), Samorodnitsky~\cite{Samorodnitsky09}, and Ellis~\cite{Ellis}), the authors and Ellis recently proved stability for all values of $\mathbb{E}[f]$, obtaining the following structure theorem, which is sharp up to an absolute constant factor.
\begin{thm}[\cite{EKL16a}]
\label{Thm:LOL}
Let $\epsilon>0$ and let $f:\{0,1\}^n \rightarrow \{0,1\}$ be a Boolean function with $\mathbb{E}[f]=\mu$, such that $I\left(f\right)\le I\left(\l_{\mu}\right)+\epsilon$, where $\l_{\mu}$ is the characteristic function of the set of the $\mu 2^n$ maximal elements in the lexicographic order. Then there exists a Boolean function $g$ such that $\{x: g(x)=1\}$ is `weakly isomorphic' to $\l_{\mu}$, and $\Pr[f(x) \neq g(x)]\le C\epsilon$, where $C$ is a universal constant. The result is sharp up to the value of the constant $C$.
\end{thm}
While Theorem~\ref{Thm:LOL} solves the `stability' question (up to an absolute constant factor), it does not tell anything about the structure of functions whose influence is larger than the minimum by $\Omega(\mu)$, let alone functions whose influence is within a constant multiplicative factor of the minimum.

\subsection{The structure of low-influence biased functions}

In~\cite{KK06}, motivated by the study of threshold phenomena in random graphs and hypergraphs, Kahn and Kalai suggested to study the structure of biased Boolean functions whose influence lies within a constant factor of the minimum possible, i.e., $I(f) \leq C \mu \log(1/\mu)$, where $\mu:=\mathbb{E}[f]$. It is clear that such functions cannot be approximated by a constant-size Junta (as even the sub-cube of measure $\mu$, whose influence is the minimum possible, cannot be approximated by a function that depends on less than $\log(1/\mu)$ coordinates). Instead, the authors of~\cite{KK06} conjectured that $f$ can be approximated by a DNF of a small width.
\begin{conjecture}[Kahn and Kalai]
\label{Conjecture: Kahn Kalai} For any $C,\epsilon,\mu>0$, there exists $w=O_{C,\epsilon}\left(\log\left(1/\mu\right)\right)$
such that the following holds. Let $f\colon\left\{ 0,1\right\} ^{n}\to\left\{ 0,1\right\} $ be a monotone Boolean function with $\mathbb{E}[f]=\mu$.
Suppose that $I\left[f\right]\le C\mu\log (1/\mu)$. Then $f$ can be $\epsilon\mu$-approximated by a DNF of width at most $w$ (i.e., a union of sub-cubes of co-dimension at most $w$).
\end{conjecture}

It should be noted that the natural adaptation of Theorem~\ref{Thm:Friedgut-Junta} to the setting of Kahn-Kalai yields the following:
\begin{thm}[Friedgut]
\label{thm:Friedgut junta} For any $C,\epsilon,\mu>0$, there exists $j=j\left(C,\epsilon,\mu\right)=\left(1/\mu\right)^{O\left(C/\epsilon\right)}$
such that the following holds. Let $f\colon\left\{ 0,1\right\} ^{n}\to\left\{ 0,1\right\} $
be a Boolean function such that $\mathbb{E}[f]=\mu$. Suppose that $I\left(f\right)\le C\mu\log(1/\mu)$. Then $f$ can be
$\epsilon\mu$-approximated by a $j$-Junta (i.e., a function that depends on at most $j$ coordinates).
\end{thm}
This result, which is tight up to the constant in the exponent, does not tell anything when $\mu$ is polynomial in $n^{-1}$, as is the case for many applications. Kahn and Kalai hoped that by replacing the `Junta approximation' with approximation by a DNF, one can obtain a meaningful structure result also for polynomially small $\mu$.

\subsection{Our results}

Unfortunately, as we show below, Conjecture~\ref{Conjecture: Kahn Kalai} is too strong, and in fact, the width of the best approximating DNF may be as large as $2^{O_{C,\epsilon}\left(\log\left(1/\mu\right)\right)}$, which (like Theorem~\ref{thm:Friedgut junta}) tells us nothing for $\mu$ polynomially small in $n$. On the other hand, we show that (a variant of) Conjecture~\ref{Conjecture: Kahn Kalai} does hold if the assumption on $I(f)$ is a bit stronger. Our main result is the following:
\begin{thm}
\label{thm:Main simplified}For any $M,\epsilon>0$, there exists $s=s\left(M,\epsilon\right)=2^{2^{O\left(M/\epsilon\right)}}$
such that the following holds. Let $f\colon\left\{ 0,1\right\} ^{n}\to\left\{ 0,1\right\} $ be a Boolean function such that $\mathbb{E}[f]=\mu$.
Suppose that $I\left(f\right)\le 2\mu(\log(1/\mu)+M)$. Then $f$ can be $\epsilon \mu$-approximated by a DNF of size $s$ (i.e., a union of $s$ sub-subes). Consequently, $f$ can be $\epsilon \mu$-approximated by a DNF of width at most $\log(1/\mu)+2^{O(M/\epsilon)}$ (i.e., a union of sub-cubes of co-dimension at most $\log(1/\mu)+2^{O(M/\epsilon)}$).
\end{thm}

\noindent \textbf{Sharpness of the result.} Theorem~\ref{thm:Main simplified} is sharp, up to the constant in the exponent. The sharpness example is the intersection of a sub-cube of co-dimension $\approx \log(1/\mu)$ with the \emph{dual tribes function} introduced by Ben-Or and Linial~\cite{BL}.

\medskip

For $w,s \in \mathbb{N}$, the \emph{tribes} function $\mathrm{Tribes}_{w,s}: \{0,1\}^{ws} \rightarrow \{0,1\}$ is defined as
\begin{equation}
\mathrm{Tribes}_{w,s}(x_1,\ldots,x_{ws}) = (x_1 \wedge \ldots \wedge x_w) \vee (x_{w+1} \wedge \ldots \wedge x_{2w}) \vee \ldots
\vee (x_{(s-1)w+1} \wedge \ldots \wedge x_{sw}),
\end{equation}
and the \emph{dual tribes} function $\mathrm{Tribes}_{w,s}^{\dagger}: \{0,1\}^{ws} \rightarrow \{0,1\}$ is defined as
\begin{equation}
\mathrm{Tribes}_{w,s}^{\dagger}(x_1,\ldots,x_{ws}) = 1-\mathrm{Tribes}_{w,s}(1-x_1,\ldots,1-x_{ws}).
\end{equation}
Now, let $w,l \in \mathbb{N}$, let $n=w2^{w}+l$, and let $f$ be the function
\[
f\left(\mathbf{x}\right)=\begin{cases}
\mathrm{Tribes}_{w,2^w}^{\dagger}\left(x_1,x_2,\ldots,x_{n-l}\right) & x_{n-l+1}=\cdots=x_{n}=1\\
0 & \mbox{Otherwise}
\end{cases}.
\]
Write $\mu=\mathbb{E}[f]$. As we show in Section~\ref{sec:Examples}, we have $I(f)=2\mu\left(\log(1/\mu)+\Theta\left(w\right)\right)$, but
$f$ cannot be $0.2\mu$-approximated by any DNF of width at most $\log\frac{1}{\mu}+\Theta\left(2^{w}\right)$. In addition, $f$ cannot
be $0.1\mu$-approximated by a DNF of size at most $2^{\Theta\left(2^{w}\right)}$. This shows the sharpness of Theorem~\ref{thm:Main simplified}, and also
provides a counterexample for Conjecture~\ref{Conjecture: Kahn Kalai}.

\medskip \noindent \textbf{Range of applicability and meaning of the result.} Theorem~\ref{thm:Main simplified} is `interesting' in the range
\begin{equation}\label{Eq:Range}
2\mu(\log(1/\mu)+ \Omega(\mu)) \leq I(f) \leq 2\mu(\log(1/\mu)+ o(\log(1/\mu))).
\end{equation}
For values of the influence smaller than the l.h.s. of~\eqref{Eq:Range}, Theorem~\ref{Thm:LOL} can be applied to get approximation by a single sub-cube. For values larger than the r.h.s. of~\eqref{Eq:Range}, i.e., $I(f) \geq c\mu \log(1/\mu)$ for $c>2$, a stronger assertion can be deduced from the Junta approximation of Friedgut's Theorem~\ref{thm:Friedgut junta}.

For $I(f)$ in the range~\eqref{Eq:Range}, on the one hand, one cannot hope for approximation by a single sub-cube, as it can be easily seen that the union of $s$ sub-cubes satisfies $I(f) = 2\mu(\log(1/\mu)+ \Theta_s(\mu))$. On the other hand, the best one can obtain using Theorem~\ref{thm:Friedgut junta} is approximation by a Junta of size $\Omega(1/\mu)$. Our Theorem~\ref{thm:Main simplified} provides approximation by a DNF whose size is much smaller, and in particular, by a \emph{constant-size DNF} for any constant $M$. Hence, it seems to be the `right' structure result one would like to achieve, at least in the range $I(f) = 2\mu(\log(1/\mu)+ \Theta(1))$.

\medskip \noindent \textbf{Our techniques.} Like the proof of Friedgut's Junta theorem, our proof makes use of discrete Fourier analysis and hypercontractivity, via the seminal KKL theorem~\cite{KKL}. In addition, we use the classical \emph{combinatorial shifting} technique~\cite{Daykin,EKR}. To be more specific, the central novel ingredient in our proof is the following lemma, that (along with its proof method) may be of independent interest.
\begin{lem}
\label{thm:Max-Inf-large}There exists an absolute constant $C_1$ such that the following holds. Let $M,\delta>0$ satisfy $M/\delta>C$, and let
$\mu \in \left(0,1-\delta\right)$. Let $f\colon\left\{ 0,1\right\} ^{n}\to\left\{ 0,1\right\} $ be a Boolean function with $\mathbb{E}[f]=\mu$, and suppose that $I(f)\le2\mu\left(\log\left(1/\mu\right)+M\right)$.
Then
\[
\max_{i\in\left[n\right]}\left\{I_{i}\left(f\right)\right\} \ge2^{-C_1 M/\delta}\mu.
\]
\end{lem}
Lemma~\ref{thm:Max-Inf-large} asserts that if the total influence of $f$ is `small', then $f$ must have an influential coordinate. For $\mu$ bounded away from 0 and 1, the Lemma follows immediately from the KKL theorem. We leverage the result to any measure $\mu$ by an inductive argument, based on the shifting technique.


\medskip \noindent \textbf{Organization of the paper.} In Section~\ref{sec:Prelim} we introduce notations to be used throughout the paper and describe the general structure of the proof of Theorem~\ref{thm:Main simplified}. In Section~\ref{sec:Lemmas} we prove the main lemmas we use in the sequel, including Lemma~\ref{thm:Max-Inf-large}. The proof of Theorem~\ref{thm:Main simplified} is presented in Section~\ref{sec:Proof}.
The sharpness examples are presented in Section~\ref{sec:Examples}, and we conclude the paper with a few open problems in Section~\ref{sec:Open}.

\medskip \noindent \textbf{Note.} Keevash and Long~\cite{KL17} have independently and simultaneously proved another version of our main theorem, with an upper bound of $2^{2^{O\left(M/\epsilon\right)^2}}$ on the size of the DNF (instead of our sharp $2^{2^{O\left(M/\epsilon\right)}}$). The methods of~\cite{KL17} is different from ours. Essentially, while we obtain our main lemma (i.e., Lemma~\ref{thm:Max-Inf-large} which asserts the existence of an influential coordinate) using combinatorial shifting and the classical KKL theorem, in~\cite{KL17} a slightly weaker version of the main lemma
is obtained using `heavier' analytic tools, including inequalities of Talagrand and Polyanskiy.

\section{Notations and Proof Overview}
\label{sec:Prelim}

\subsection{Notations}

First, for sake of completeness we give the formal definition of a DNF and its width and size.

A \emph{literal} is either a variable $x_{i}$ or its negation. A \emph{term }is an AND of literals, and a \emph{DNF} is an OR of terms. E.g., the following $\left(x_{1}\wedge\neg x_{2}\right)\vee\left(x_{2}\wedge x_{3}\right)\vee\left(x_{3}\wedge\neg x_{4}\wedge x_{5}\right)$ is a DNF formula. Let $D=T_{1}\vee T_{2}\vee\cdots\vee T_{s}$ be a DNF. The \emph{size} of $D$ is the amount of literals in $D$ (i.e., $s$). The \emph{width} of $D$ is the maximal number of literals in a term of $D$. (So, the above DNF has size 3 and width 3). We identify a DNF on $n$ variables with the Boolean function $f:\{0,1\}^n \rightarrow \{0,1\}$ defined as $f(x_1,\ldots,x_n)=1$ if and only if $(x_1,\ldots,x_n)$ satisfies the formula. Note that each term corresponds to a subcube, a DNF of
size $s$ corresponds to the characteristic function of the union of $s$ subcubes, and its width is the maximal co-dimension of a sub-cube that corresponds to one of its terms.

\medskip

Throughout the paper, $[n]$ denotes the set $\{1,2,\ldots,n\}$, and $C,c,C_i$ denote universal constants.
$f$ will be denote a Boolean function, i.e., $f\colon\left\{ 0,1\right\} ^{n}\to\left\{ 0,1\right\} $, and $\mathbb{E}[f]$ will be denoted by $\mu(f)$ or simply by $\mu$. We will assume throughout that $I_1(f)$ is the maximal influence of $f$. (There is no loss of generality in this assumption, as we can always reorder the coordinates of $f$.) We let $s_{f}\left(\epsilon\right)$ be the minimal size of a DNF that $\epsilon\mu\left(f\right)$-approximates $f$, and define $M$ to be such that
\[
I\left(f\right)=2\mu\left(f\right)\left(\log\left(1/\mu\left(f\right)\right)+M\right).
\]
(Note that $M \geq 0$ by Theorem~\ref{Thm:Iso-weak}.)

\medskip

The proof of Theorem~\ref{thm:Main simplified} will use an inductive approach, for which we will persistently use the following notations.
For a function $f$, we let $f_{1},f_{0}\colon\left\{ 0,1\right\} ^{n-1}\to\left\{ 0,1\right\} $ be the Boolean functions defined by
\[
f_{1}(x_1,\ldots,x_{n-1})=f\left(1,x_1,\ldots,x_{n-1}\right), \qquad f_{0}\left(x_1,\ldots,x_{n-1}\right)=f\left(0,x_1,\ldots,x_{n-1}\right).
\]
We write $\mu_{1}=\mu\left(f_{1}\right)$ and $\mu_{0}=\mu\left(f_{0}\right)$. Similarly, we let $M_{1},M_{0} \geq 0$ be the numbers satisfying
\[
I\left(f_{1}\right)=2\mu_{1}\left(\log\frac{1}{\mu_{1}}+M_{1}\right), \qquad I\left(f_{0}\right)=2\mu_{0}\left(\log\frac{1}{\mu_{0}}+M_{0}\right).
\]
We will use the following simple (and well-known) fact:
\begin{equation}\label{Eq:InfInd}
I(f) = \frac{1}{2}(I(f_1)+I(f_0)) + I_1(f).
\end{equation}

\subsection{Proof overview}

The inductive approach of the proof is based on the following simple observation:
\begin{obs}
If $f_{1}$ can be $\epsilon_{1}\mu_{1}$-approximated by a DNF of size $s_{1}$, and if $f_{0}$ can be $\epsilon_{0}\mu_{0}$-
approximated by a DNF of size at most $s_{0}$, then $f$ can be $\frac{\epsilon_{1}\mu_{1}+\epsilon_{0}\mu_{0}}{2}$-approximated
by a DNF of size at most $s_{1}+s_{0}$.
\end{obs}
It follows that if $\epsilon_1,\epsilon_2$ are chosen such that $\epsilon_{1}\mu_{1}+\epsilon_{0}\mu_{0}=2\epsilon\mu$, then we have
\begin{equation}
s_{f}\left(\epsilon\right)\le s_{f_{1}}\left(\epsilon_{1}\right)+s_{f_{0}}\left(\epsilon_{0}\right).\label{eq:basic observation.}
\end{equation}
We perform the inductive step, rearranging the coordinates such that coordinate 1 is the \emph{most influential} one. We distinguish between three cases:
\begin{itemize}
\item \textbf{Both $\min\left\{ \mu_{0},\mu_{1}\right\} $ and $I_{1}\left(f\right)$ are `not too small'.} In this case, we use~\eqref{eq:basic observation.} to combine an $\epsilon_1$-approximation of $f_1$ with an $\epsilon_0$-approximation of $f_0$ into an approximation of $f$. We choose $\epsilon_1,\epsilon_0$ in such a way that $\frac{M_{0}}{\epsilon_{0}}=\frac{M_{1}}{\epsilon_{1}}$, so that the sizes of the DNFs approximating $f_1$ and $f_0$ will be roughly equal. While this step doubles the size of the approximating DNF (compared to those approximating $f_1,f_0$), we show that $\epsilon_1/M_1,\epsilon_0/M_0$ which replace $\epsilon/M$ are larger than $\epsilon/M$ by at least a fixed amount (which depends on $\epsilon$), and so, the number of required `doubling' steps will be eventually bounded.

\item \textbf{$\min\left\{ \mu_{0},\mu_{1}\right\} $ is `small'.} Of course, we may assume w.l.o.g. that $\mu_0$ is small. In this case, it is better to approximate $f_{0}$ by the constant $0$ function, rather than waste any subcubes on it. This step does not increase the size of the DNF, but seems to make the approximation worse. We show that nevertheless, the proof can go through, exploiting the (relatively) large influence of the first coordinate.

\item \textbf{$I_{1}\left(f\right)$ is `small'.} We conclude the proof by showing that this case is impossible, as any function with a small total influence must have an influential coordinate. This is the main part of the proof, encapsulated in Lemma~\ref{thm:Max-Inf-large}.
\end{itemize}

\section{The Central Lemmas}
\label{sec:Lemmas}

In this section we prove the two central lemmas needed for the proof of Theorem~\ref{thm:Main simplified}.

\subsection{Low-influence functions have an influential coordinate}

In this subsection we prove Lemma~\ref{thm:Max-Inf-large}. The proof requires two different types of tools -- Fourier-theoretic and combinatorial.

The Fourier-theoretic tool we use is the classical KKL theorem~\cite{KKL}. (The version presented here is taken from Section 9.6 of~\cite{O'D14}, where it is called `the KKL edge isoperimetric theorem').
\begin{thm}[Kahn, Kalai, and Linial]\label{Thm:KKL}
Let $f:\{0,1\}^n \rightarrow \{0,1\}$ be a non-constant Boolean function, and let $\tilde{I}(f)=\frac{I(f)}{4\mathbb{E}[f](1-\mathbb{E}[f])}$. Then
\[
\max_{1 \leq i \leq n} I_i(f) \geq \frac{9}{\tilde{I}(f)^2} 9^{-\tilde{I}(f)}.
\]
\end{thm}

The combinatorial tool is the classical \emph{shifting operators} $\s_{ST}$, introduced by Erd\H{o}s, Ko, and Rado~\cite{EKR} and developed by Daykin~\cite{Daykin} and others.

For $\mathbf{x} \in \{0,1\}^n$ and $S \subset \left[n\right]$, we write $\mathbf{x}_{S}=1$ if $x_{i}=1$ for all $i\in S$. Similarly, we write $\mathbf{x}_{S}=0$ if $x_{i}=0$ for all $i\in S$. We also write $1_{S} \in \{0,1\}^n$ for the indicator vector of $S$ (i.e., $1_{S}(i)=1$ if and only if $i \in S$).
\begin{defn}
Let $f\colon\left\{0,1\right\} ^{n} \rightarrow \{0,1\}$ be a Boolean function, and let $S,T\subseteq\left[n\right]$ be disjoint sets. The `shifted function'
$\s_{ST}\left(f\right)$ is defined by setting
\[
\s_{ST}\left(f\right)\left(\mathbf{x}\right):=\begin{cases}
f\left(\mathbf{x}\right)\wedge f\left(\mathbf{x}\oplus1_{S\cup T}\right) & \mbox{if }\mathbf{x}_{S}=1\mbox{ and }\mathbf{x}_{T}=0\\
f\left(\mathbf{x}\right)\vee f\left(\mathbf{x}\oplus1_{S\cup T}\right) & \mbox{if }\mathbf{x}_{T}=1\mbox{ and }\mathbf{x}_{S}=0\\
f\left(\mathbf{x}\right) & \mbox{otherwise.}
\end{cases}
\]
\end{defn}
A more intuitive definition of the shifting operator $\s_{ST}$ is as follows. Write $f=1_{A}$ for $A \subset \{0,1\}^n$. The operator $\s_{ST}$ takes all elements $\mathbf{x}\in A$ such that $\mathbf{x}_{S}=1$, $\mathbf{x}_{T}=0$, and $\mathbf{x}\oplus1_{S\cup T}\notin A$, and replaces them with $\mathbf{x}\oplus1_{S\cup T}$. All other elements of $A$ are left unchanged.

The shifting operators will be useful for us due to the following well-known Lemma. 
\begin{lem}
\label{lem:shifting}Let $f\colon\left\{ 0,1\right\} ^{n}\to\left\{ 0,1\right\} $
be a Boolean function of measure $\mu\left(f\right)\le\frac{1}{2}$.
Write
\[
f^{0}=\s_{\varnothing\left\{ 1\right\} }\circ\s_{\varnothing\left\{ 2\right\} }\circ\cdots\circ\s_{\varnothing\left\{ n\right\} }\left(f\right),
\]

\[
f^{1}=\s_{\left\{ n\right\} \left\{ 1\right\}}\circ\s_{\left\{ n-1\right\} \left\{ 1\right\}}\circ\cdots\circ\s_{\left\{ 2\right\} \left\{ 1\right\}}\left(f^{0}\right),
\]

\[
\vdots
\]

\[
f^{n}=\s_{\left\{ n,\ldots,2\right\} \left\{ 1\right\}}\left(f^{n-1}\right).
\]
Then:
\begin{itemize}
\item $I_{i}\left(f^{n}\right) \le I_{i}\left(f\right)$ for any $i\ge2$.
\item $I\left(f^{n}\right)\le I\left(f\right).$
\item The function $f^n$ satisfies $f^n(0,x_2,x_3,\ldots,x_n)=0$ for all $x_2,\ldots,x_n$.
\end{itemize}
\end{lem}

Now we are ready to present the proof of Lemma~\ref{thm:Max-Inf-large}. For convenience, we recall the statement of the Lemma.

\medskip \noindent \textbf{Lemma~\ref{thm:Max-Inf-large}.}
There exists an absolute constant $C_1$ such that the following holds. Let $M,\delta>0$ satisfy $M/\delta>C$, and let
$\mu \in \left(0,1-\delta\right)$. Let $f\colon\left\{ 0,1\right\} ^{n}\to\left\{ 0,1\right\} $ be a Boolean function with $\mathbb{E}[f]=\mu$, and suppose that $I(f)\le2\mu\left(\log\left(1/\mu\right)+M\right)$.
Then
\[
\max_{i\in\left[n\right]}\left\{I_{i}\left(f\right)\right\} \ge2^{-C_1 M/\delta}\mu.
\]

\begin{proof}
Suppose first that $\mu\ge\frac{1}{4}$. In this case, we have
\[
\tilde{I}(f)=\frac{I(f)}{4\mathbb{E}[f](1-\mathbb{E}[f])} \leq \frac{4}{3} \cdot (M/\delta),
\]
and thus, the assertion follows immediately from Theorem~\ref{Thm:KKL}.

\medskip

Now suppose that $\mu\le\frac{1}{4}$. Let $i \in \mathbb{N}$ be such that $\mu \in [2^{-1-i},2^{-i})$. The proof will proceed by induction on $i$. Let $f^{n}$ be as in Lemma \ref{lem:shifting}, and define $f_{1}^{n},f_{0}^{n}\colon\left\{0,1\right\}^{n-1} \rightarrow \{0,1\}$
by $f_{1}^{n}\left(\mathbf{x}\right)=f^{n}\left(1,\mathbf{x}\right)$ and $f_{0}^{n}\left(\mathbf{x}\right)=f^{n}\left(0,\mathbf{x}\right)$.
Recall that by Lemma \ref{lem:shifting}, we have $f_{0}^{n}\left(\mathbf{x}\right) \equiv 0$.
Thus,
\[
2\mu\left(\log(1/\mu)+M\right)\ge I\left(f^{n}\right)=\frac{1}{2}I\left(f_{1}^{n}\right)+\frac{1}{2}I\left(f_{0}^{n}\right)+I_{1}\left(f^{n}\right)=2\mu+\frac{1}{2}I\left(f_{1}^{n}\right),
\]
where the leftmost equality follows from~\eqref{Eq:InfInd}. Write $\mu_{1}=2\mu=\mu\left(f_{1}^{n}\right)$. We obtain
\[
I\left(f_{1}^{n}\right) \leq 2\mu_{1}\left(\log (1/\mu_{1})+M\right).
\]
By the induction hypothesis, the maximal influence of $f_{1}^{n}$ is at least $2^{-C_1M/\delta}\mu_{1}$. This implies that $I_{i}\left(f^{n}\right)\ge2^{-C_1M/\delta}\frac{\mu_{1}}{2}$ for some $i\ge2$. By Lemma \ref{lem:shifting}, it follows that $I_{i}\left(f\right)\ge2^{-C_1M/\delta}\mu$.
This completes the proof.
\end{proof}

\subsection{The effect of an influential coordinate on the restricted functions in the induction process}

In this subsection we suppose w.l.o.g. that $I_{1}\left(f\right)$ is the maximal influence of $f$. By Lemma~\ref{thm:Max-Inf-large}, $I_1(f)$ is `not very small'. We show that in this case, when we perform the induction process on the first coordinate (as described in Section~\ref{sec:Prelim}), the influences $I(f_1)$ and $I(f_0)$ are, on average, `closer to the minimum' than $I(f)$. On the intuitive level, this is apparent in view of~\eqref{Eq:InfInd}, but we need a quantitative result. The `advantage' we obtain here will be crucial in the inductive step of Theorem~\ref{thm:Main simplified}, both in the case where $\mu_0$ is small (where it will compensate for a looser approximation, resulting from approximating $f_0$ by the zero function), and in the case where $\mu_0,\mu_1$, and $I_1(f)$ are all large (where it will allow to bound the number of steps that double the size of the approximating DNF).

\medskip

The following lemma was proved by Ellis~\cite{Ellis}.
\begin{lem}
\label{lem:Win if the influence is medium}There exists an absolute constant $c$ such that the following holds. Let $\zeta>0$ and let $f\colon\left\{ 0,1\right\} ^{n}\to\left\{ 0,1\right\}$ be a Boolean function. If $\min(I_{1}\left(f\right),\mu_0,\mu_1) \ge \zeta\mu$, then
\[
2M\mu-M_{1}\mu_{1}-M_{0}\mu_{0}\ge c\zeta\mu.
\]
\end{lem}
We prove a similar result in the case where $\mu_{0}$ (or, equivalently, $\mu_1$) is small.
\begin{lem}
\label{lem:Win if the influence is super large} Let $C_3>0$. Suppose that
\[
\min\left\{ \mu_{0},\mu_{1}\right\} \le2^{-C_3}\mu.
\]
Then
\[
2M\mu-M_{1}\mu_{1}-M_{0}\mu_{0}\ge\left(C_3-1\right) \min(\mu_{0},\mu_1).
\]
\end{lem}
\begin{proof}
We assume w.l.o.g. that $\mu_{0}\le\mu_{1}$. The lemma follows from a straightforward computation:
\begin{align*}
2M\mu-M_{1}\mu_{1}-M_{0}\mu_{0} & =I(f)-2\mu\log\frac{1}{\mu}-\frac{1}{2}\left(I(f_1)-2\mu_{1}\log\frac{1}{\mu_{1}}\right)\\
 & -\frac{1}{2}\left(I(f_0)-2\mu_{0}\log\frac{1}{\mu_{0}}\right)\\
 & =I_{1}(f)+\mu_{1}\log\frac{1}{\mu_{1}}+\mu_{0}\log\frac{1}{\mu_{0}}-2\mu\log\frac{1}{\mu}\\
 & \ge\mu_{1}-\mu_{0}+\mu_{1}\log\frac{1}{\mu_{1}}+\mu_{0}\log\frac{1}{\mu_{0}}-2\mu\log\frac{1}{\mu}\\
 & =\mu_{1}\log\frac{2}{\mu_{1}}+\mu_{0}\log\frac{1}{2\mu_{0}}-2\mu\log\frac{1}{\mu}\\
 & \ge\mu_{1}\log\frac{1}{\mu}+\mu_{0}\log\frac{1}{2\mu_{0}}-2\mu\log\frac{1}{\mu}\\
 & =\mu_{0}\left(\log\frac{1}{2\mu_{0}}-\log\frac{1}{\mu}\right)\\
 & =\mu_{0}\log\left(\frac{\mu}{2\mu_{0}}\right)\ge\left(C_3-1\right)\mu_{0}.
\end{align*}
\end{proof}

\section{Proof of the Main Theorem}
\label{sec:Proof}

\begin{defn}
Let $\mu\in\left(0,1\right),\epsilon>0$, and $n \in \mathbb{N}$. We define $\tilde{s}\left(\mu,\epsilon,n\right)$ to be the smallest integer
such that the following holds. Let $f\colon\left\{ 0,1\right\} ^{n}\to\left\{ 0,1\right\} $
be a Boolean function, and write
\[
I\left(f\right)=2\mu\left(\log\left(\frac{1}{\mu}\right)+M\right).
\]
Then $f$ can be $\epsilon M\mu$-approximated by a DNF of size $\tilde{s}\left(\mu,\epsilon,n\right)$.

We also write $\tilde{s}\left(\epsilon\right)$ for the supremum of $\tilde{s}\left(\mu,\epsilon,n\right)$ over all $\mu\in\left(0,1\right)$,
and all $n\in\n$.
\end{defn}

It is clear that in order to prove Theorem \ref{thm:Main simplified}, it is sufficient to show that
\begin{equation}\label{Eq:Main1}
\tilde{s}\left(\epsilon\right) \leq 2^{2^{O\left(\frac{1}{\epsilon}\right)}}
\end{equation}
for any $\epsilon>0$. Throughout this section, we assume w.l.o.g that $I_1(f)$ is the maximal influence of $f$, and that $\mu_{0}\le\mu_{1}$.

\medskip

First, we show that one can assume w.l.o.g. that $\epsilon<C_4$ for a constant $C_4$. This follows immediately from the stability version of Theorem~\ref{Thm:Iso-weak} proved by the first author~\cite{Ellis}.
\begin{thm}[\cite{Ellis}]
\label{thm:Ellis} There exist an absolute constant $c'>0$ such that the following holds. Let $f\colon\left\{ 0,1\right\} ^{n}\to\left\{ 0,1\right\} $
be a Boolean function with $\mathbb{E}[f]=\mu$, and let $\epsilon>0$. Suppose that
\[
I(f)\le \mu\left(\log(1/\mu)+c'\epsilon\log\left(1/\epsilon\right)\right).
\]
Then $f$ can be $\epsilon \mu$-approximated by a subcube.
\end{thm}

\begin{lem}
\label{lem:Ellis-1}There exists an absolute constant $C_4$ such that for all $\epsilon>C_4$,
\[
\tilde{s}\left(\epsilon\right)=1.
\]
\end{lem}
\begin{proof}
Let $\epsilon>C_4$ for $C_4$ to be specified below, and let $f\colon\left\{ 0,1\right\} ^{n}\to\left\{ 0,1\right\} $ be a
Boolean function. Write $I\left(f\right)=2\mu\left(\log\frac{1}{\mu}+M\right)$. We have to show that $f$ can be $\epsilon M\mu$-approximated by a subcube. If $M\epsilon\ge1$, then $f$ can be approximated by the constant 0 function. Thus, we may assume that $M\le\frac{1}{C_4}\le1$,
provided that $C_4\ge1$. By Theorem~\ref{thm:Ellis}, there exists $c'>0$, such that $f$ can be $c'\frac{M}{\log\left(1/M\right)}\mu$-approximated
by a subcube. Hence, $f$ can be $\epsilon M\mu$-approximated by a subcube provided that $C_4$ is sufficiently large. This completes the proof.
\end{proof}

Now we present the main part of the inductive argument.
We show that there exists $C_{5}>0$ such that in any step of the inductive process, one of the following alternatives must occur:
\begin{enumerate}
\item Either there exists some $\mu_{1}\ge\mu$, such that
\[
\tilde{s}\left(\mu,\epsilon,n\right)\le\tilde{s}\left(\mu_{1},\epsilon,n-1\right),
\]

\item Or
\[
\tilde{s}\left(\mu,\epsilon,n\right)\le2\tilde{s}\left(\epsilon+2^{-C_{5}/\epsilon}\right).
\]
\end{enumerate}

This will follow immediately from combination of two claims:
\begin{claim}
\label{Claim: mu-small}There exists $C_{6}>0$
such that the following holds. If $\mu_{0}\le2^{-C_{6}/\epsilon}\mu$,
then $f$ can be $\epsilon M\mu$-approximated by a DNF of size at
most $\tilde{s}\left(\mu_{1},\epsilon,n-1\right)$.
\end{claim}

\begin{claim}
\label{Claim:medium mu_0} Let $C_{6}>0$ be some constant, and suppose
that $\mu_{0}\ge2^{-C_{6}/\epsilon}\mu$. Then $f$ can be $\epsilon M\mu$-approximated
by a DNF of size at most $2\tilde{s}\left(\epsilon+2^{-C_{5}/\epsilon}\right)$,
provided that $C_{5}$ is large enough. \end{claim}

\begin{proof}[Proof of Claim \ref{Claim: mu-small}]
Note that $f_{1}$ can be $\epsilon M_{1}\mu_{1}$-approximated
by a DNF of size $s':=s\left(\epsilon,\mu_{1},n\right)$, say $T_{1}\vee T_{2}\vee\cdots\vee T_{s'}$.
This implies that $f$ can be $\left(\frac{1}{2}\epsilon M_{1}\mu_{1}+\frac{1}{2}\mu_{0}\right)$-approximated
by the DNF $\left(1\wedge T_{1}\right)\vee\left(1\wedge T_{1}\right)\vee\cdots\vee\left(1\wedge T_{s'}\right)$.
The claim will follow once we show that $\frac{1}{2}\epsilon M_{1}\mu_{1}+\frac{1}{2}\mu_{0}\le\epsilon M\mu$,
provided that $C_{6}$ is large enough.

We may assume that $M\le\frac{1}{\epsilon}$, for otherwise $f$ is
$\epsilon M\mu$-approximated by the constant $0$ function. By Lemma
\ref{lem:Ellis-1}, there exists an absolute constant $C_4$, such
that $\tilde{s}\left(\mu,\epsilon,n\right)=1$ provided that $\epsilon\ge C_4$.
Thus, we may assume that $\epsilon\le C_4$. By Lemma \ref{lem:Win if the influence is super large},
\begin{equation}
2M\mu-M_{1}\mu_{1}-M_{0}\mu_{0}\ge\left(C_6/\epsilon-1\right)\mu_{0}\ge\left(M+\frac{C_6-1-C_4}{\epsilon}\right)\mu_{0}
\ge\left(M+2/\epsilon\right)\mu_{0},\label{eq:1},
\end{equation}
 where the last inequality holds provided that $C_6$ is large enough.
Substituting $\mu=\frac{\mu_{1}+\mu_{0}}{2}$ in (\ref{eq:1}), we
obtain
\[
\left(M-M_{1}\right)\mu_{1}+M\mu_{0}\ge\left(M-M_{1}\right)\mu_{1}+\left(M-M_{0}\right)\mu_{0}\ge\left(\frac{2}{\epsilon}+M\right)\mu_{0}.
\]
Rearranging yields
\begin{equation}
M_{1}\mu_{1}+\frac{2\mu_{0}}{\epsilon}\le\mu_{1}M\le2M\mu.\label{eq:2}
\end{equation}
We now multiply (\ref{eq:2}) by $\frac{\epsilon}{2}$ to finish
the proof of the claim.
\end{proof}

\begin{proof}[Proof of Claim \ref{Claim:medium mu_0}]
As mentioned before, we may assume that $M\le\frac{1}{\epsilon}.$
By Lemma \ref{thm:Max-Inf-large}, there exists $C_{1}>0$, such that $I_{1}\left(f\right)\ge2^{-C_{1}/\epsilon}\mu$.
By Lemma \ref{lem:Win if the influence is medium}, there exists $c>0$, such that
\begin{equation}
2M\mu-M_{1}\mu_{1}-M_{0}\mu_{0}\ge c2^{-\max\left\{ C_{1},C_{6}\right\} /\epsilon}\mu.\label{eq:3}
\end{equation}
Write $B=c2^{-\max\left\{ C_{1},C_{6}\right\} /\epsilon}$ and $\epsilon'=\frac{2M}{2M-B}\epsilon$.
Let $D_{1}$ be the DNF of size at most $\tilde{s}\left(\mu_{1},\epsilon',n-1\right)$
that $\epsilon'M_{1}\mu_{1}$-approximates $f_{1}$, and let $D_{0}$
be the DNF of size at most $\tilde{s}\left(\mu_{0},\epsilon',n-1\right)$
that $\epsilon'M_{0}\mu_{0}$-approximates $f_{0}$. Let $(x_1 \wedge D_{1}) \vee (\neg{x_1} \wedge D_{0})$ be the DNF defined by adding the literal $x_1$
to each term of $D_1$, adding the literal $\neg x_1$ to each term of $D_0$, and conjuncting the resulting DNFs. The size of the resulting DNF is at
most $\tilde{s}\left(\mu_{1},\epsilon',n-1\right)+\tilde{s}\left(\mu_{0},\epsilon',n-1\right) \leq 2\tilde{s}(\epsilon')$, and it clearly
$\frac{1}{2}\epsilon'M_{0}\mu_{0}+\frac{1}{2}\epsilon' M_{1}\mu_{1}$-approximates
$f$. By (\ref{eq:3}) we have
\[
\frac{1}{2}\epsilon'M_{0}\mu_{0}+\frac{1}{2}\epsilon' M_{1}\mu_{1} \leq \epsilon M\mu,
\]
and thus, $f$ can be $\epsilon M\mu$-approximated by a DNF of size at most $2\tilde{s}(\epsilon')$.
Finally, provided that $C_{5}$ is large enough, we have $\epsilon'=\frac{2M}{2M-B}\epsilon\ge2^{-C_{5}/\epsilon}+\epsilon$.
This completes the proof.
\end{proof}

We are now ready to prove Theorem \ref{thm:Main simplified}.
\begin{proof}[Proof of Theorem \ref{thm:Main simplified}]
By Lemma \ref{lem:Ellis-1}, there exists an absolute constant $C_{4}$
such that $\tilde{s}\left(\epsilon\right)=1$ for any $\epsilon\ge C_{4}$.
Let $C_{5}$ be the constant from Claim~\ref{Claim:medium mu_0}. By combination of Claims~\ref{Claim: mu-small} and~\ref{Claim:medium mu_0},
a simple inductive argument implies that $\tilde{s}\left(\epsilon\right)\le2\tilde{s}\left(\epsilon+2^{-C_{5}/\epsilon}\right)$.
Applying this inequality repeatedly $C_{4}2^{C_{5}/\epsilon}$
times, we obtain $\tilde{s}\left(\epsilon\right)\le C_{4}2^{2^{C_{5}/\epsilon}}$.
This completes the proof.
\end{proof}

\section{\label{sec:Examples}Sharpness Example}

In this section we present in detail the sharpness example for Theorem~\ref{thm:Main simplified}, that is also a counterexample to
Conjecture~\ref{Conjecture: Kahn Kalai}.

The example is based on the classical `tribes' function that was introduced by Ben-Or and Linial~\cite{BL} in 1985 and is known to be an extremal example for numerous results on Boolean functions.
\begin{defn}
The tribes function of width $w$ and size $s$ is defined by
\[
\mathrm{Tribes}_{w,s}\left(\mathbf{x}\right)=\left(x_{1}\wedge x_{2}\wedge\cdots\wedge x_{w}\right)\vee\left(x_{w+1},\ldots,x_{2w}\right)\vee\cdots\vee\left(x_{\left(s-1\right)w+1}\wedge\cdots\wedge x_{sw}\right).
\]
The dual of the tribes function is the function $\mathrm{Tribes}_{w,s}^{\dagger}$ defined by
\[
\mathrm{Tribes}_{w,s}^{\dagger}\left(\mathbf{x}\right)=1-\mathrm{Tribes}_{w,s}\left(\overline{\mathbf{x}}\right),
\]
where $\overline{\mathbf{x}}$ is the vector obtained from $\mathbf{x}$ by flipping all of its coordinates.
\end{defn}
We will use two well-known results: one regarding properties of the dual tribes function, and another regarding approximation by DNFs.
\begin{thm}[\cite{OW07}]
\label{thm:O'Donnell-Wimmer}Let $w\in\mathbb{N}$, and let $f=\mathrm{Tribes}_{w,2^{w}}^{\dagger}\left(\mathbf{x}\right)$.
Then $I\left(f\right)=\Theta\left(w\right)$, and $f$ cannot be $0.2$-approximated by a DNF of width at most $\frac{1}{3}2^{w}$.
\end{thm}

\begin{lem}
\label{lem:width at most log size}Let $D$ be a DNF of size $s$. Then it can be $s2^{-w}$-approximated by a DNF of width at most $w$
and of size at most $s$.
\end{lem}
\begin{proof}
Remove from $D$ all terms that contain more then $w$ literals to obtain a new DNF, $D'$. A union bound implies that $D'$ $s2^{-w}$-approximates
$D$. This completes the proof.
\end{proof}

Now we are ready to present our tightness example for Theorem~\ref{thm:Main simplified}.
\begin{prop}
Let $w,l \in \mathbb{N}$, let $n=w2^{w}+l$, let $f$ be the function
\[
f\left(\mathbf{x}\right)=\begin{cases}
\mathrm{Tribes}_{w,2^w}^{\dagger}\left(x_1,x_2,\ldots,x_n-l \right) & x_{n-l+1}=\cdots=x_{n}=1\\
0 & \mbox{Otherwise}
\end{cases},
\]
and write $\mu=\mathbb{E}\left[f\right]$. Then on the one hand, $I\left(f\right)=2\mu\left(\log(1/\mu)+\Theta\left(w\right)\right)$.
On the other hand, $f$ cannot be $0.2\mu$-approximated by any DNF of width at most $\log(1/\mu)+\Theta\left(2^{w}\right)$. As a consequence, $f$ cannot
be $0.1\mu$-approximated by a DNF of size at most $2^{\Theta\left(2^{w}\right)}$. \end{prop}

\begin{proof}
Suppose that $D$ is a DNF that $0.2\mu$-approximates $f$. Without
loss of generality, we may assume that all the terms of $D$ contain
the variables $x_{n-l+1},\ldots,x_{n}$. Let $D'$ be the DNF obtained
from $D$ by removing the variables $x_{n-l+1},\ldots,x_{n}$ from all its terms. Then the DNF $D'$ is $\left(0.2\cdot2^{l}\cdot\mu\left(f\right)\right)$-approximated
by the function $\mathrm{Tribes}_{w,2^{w}}^{\dagger}$. Theorem \ref{thm:O'Donnell-Wimmer}
implies that the width of $D'$ is at least $\Theta\left(w\right)$.
This completes the proof of the first part of the corollary. The {}``as a consequence'' statement follows immediately from Lemma~\ref{lem:width at most log size}. This completes the proof.
\end{proof}

\section{Open Problems}
\label{sec:Open}

We conclude this paper with a few open problems.

\medskip

\noindent \textbf{Functions with influence within a constant multiplicative factor from the minimum possible.} While Theorem~\ref{thm:Main simplified} describes rather precisely the structure of functions with $I(f) \leq 2\mu(f)(\log(1/\mu(f))+o(\log(1/\mu(f))$, the result we obtain for $I(f) = c \mu(f)\log(1/\mu(f))$ is not stronger than what one can get from Friedgut's Junta theorem. In~\cite{KK06}, Kahn and Kalai presented several conjectures on the structure of such functions (one of them is Conjecture~\ref{Conjecture: Kahn Kalai} above), and it will be interesting to see whether our techniques can be helpful in addressing them.

\medskip \noindent \textbf{Biased functions with respect to a biased measure.} As described in the introduction, structure theorems for balanced functions with respect to a biased measure $\mu_p$ on the discrete cube were studied in numerous papers (e.g.,~\cite{Bourgain99,Bourgain-Kalai,Friedgut-SAT,Hatami12}). Our paper deals with biased functions with respect to the uniform measure. Hence, the next natural goal in this respect is to study biased functions with respect to a biased measure.

To this end, one may use the classical techniques for reduction from the biased measure to the uniform measure (see, e.g.,~\cite{Friedgut-Kalai,Keller-Reduction})  to obtain a biased-measure version of Theorem~\ref{thm:Main simplified}. However, this version holds only when both $p$ and $\mu$ are not very small. It seems that more powerful techniques will be needed to address the (biased function, biased measure) case.

\medskip \noindent \textbf{A sharper approximation by a Junta?} We tend to believe that Theorem~\ref{thm:Main simplified} can be strengthened into an improved `approximation by Junta' theorem. Specifically, the following conjecture seems reasonable:
\begin{conjecture}\label{Conj:Junta}
For any $M,\epsilon>0$, any function $f\colon\left\{ 0,1\right\} ^{n}\to\left\{ 0,1\right\} $ that satisfies $I(f)\le 2\mu(\log(1/\mu)+M)$ can be $\epsilon \mu$-approximated by a function $g$ that depends on at most $O(\log(1/\mu) \cdot 2^{M/\epsilon})$ coordinates.
\end{conjecture}
For $I(f) = c \mu(f)\log(1/\mu(f))$, Conjecture~\ref{Conj:Junta} is no better than the Junta theorem, but in the range $I(f) = 2\mu(f)(\log(1/\mu(f))+M$ with $M=o(\log(1/\mu(f)))$, the size of Junta it yields is much smaller. In particular, when $M$ is constant, it becomes as small as the clearly optimal $\Theta(\log(1/\mu))$.

\section*{Acknowledgements}

We are grateful to David Ellis for communicating to us the independent work~\cite{KL17}, and to Peter Keevash and Eoin Long (the authors of~\cite{KL17}) for useful suggestions.


\end{document}